\documentclass{amsart}
\usepackage{amsmath}%
\usepackage{amsfonts}%
\usepackage{amssymb}%
\usepackage{graphicx}
\newtheorem{theorem}{Theorem}
\theoremstyle{plain}

\newtheorem{lemma}{Lemma}

\newtheorem{remark}{Remark}

\numberwithin{equation}{section}

\begin{document}
\title[]{On the sum of powered distances to certain sets of points on the circle}
\author{Nikolai Nikolov and Rafael Rafailov}
\address
{Institute of Mathematics and Informatics,
Bulgarian Academy of Sciences,
Acad. G. Bonchev 8, 1113 Sofia,
Bulgaria}\email{nik@math.bas.bg}

\address
{Sofia High School of Mathematics, Iskar 61, 1000 Sofia,
Bulgaria}
\email{rafael.rafailov@yahoo.com}


\dedicatory{}
\begin{abstract}

In this paper we consider an extremal problem in geometry.
Let $\lambda$ be a real number and $A$, $B$ and $C$ be arbitrary points on the unit circle $\Gamma$. We give full characterization of the extremal behavior of the function $f(M,\lambda)=MA^\lambda+MB^\lambda+MC^\lambda$, where $M$ is a point on the unit circle as well.
We also investigate the extremal behavior of $\sum_{i=1}^nXP_i$, where $P_i, i=1,\ldots,n$ are the vertices of a regular $n$-gon and $X$ is a point on $\Gamma$, concentric to the circle circumscribed around $P_1\ldots P_n$.
We use elementary analytic and purely geometric methods in the proof.
\end{abstract}

\maketitle

\section{Introduction}

The question of placing electrical charges on a sphere in such a way that the potential energy of the system obtains its extremal values has long been of importance to physics. Problems of the above kind have also been considered in classical potential theory.\\

The planar case of the above question is answered by the general solution of placing $n$ points $M_i$, $i=1,\ldots, n$ on the unit circle in such a way as to obtain the extreme values of the sum $$\sum_{0\leq i<j\leq n}M_iM_j^\lambda,$$ where $\lambda$ is a given real number, the concrete case being $\lambda=-1$.

There is a growing amount of literature on the above problem, which is derived as a discrete analog of questions studied in classical potential theory.
This has led to the development of the problem of placing a point $M$ on the unit circle in such a way to obtain the extremal values of $$\sum_{i=1}^nMM_i^\lambda$$ for a given point set $M_i$, $i=1,\ldots,n$. This has proven to be a difficult question and in general remains open.\\
Here we consider the case $n=3$ and prove the following

\begin{theorem}

Let $\lambda$ be a real number and $A$, $B$ and $C$ be arbitrary points on the unit circle $\Gamma$.
\begin{enumerate}
\item $\lambda<0$. There is always a point $M$ on $\Gamma$ such that $f(M,\lambda)\leq 2+2^\lambda$.
\item $\lambda\in[0;2]$. There is always a point $M$ on $\Gamma$ such that $f(M,\lambda)\geq 2+2^\lambda$.
\item $\lambda\in(2;4)$. There is always a point $M$ on $\Gamma$ such that $f(M,\lambda)\geq2\sqrt3^\lambda$.
\item $\lambda\geq4$.    There is always a point $M$ on $\Gamma$ such that $f(M,\lambda)\geq 2+2^\lambda$.\\
The above bounds are sharp if and only if $A$, $B$ and $C$ are the vertices of an equilateral triangle.
\end{enumerate}
\end{theorem}

Or, equivalently
\begin{enumerate}
\item $$\min_{A,B,C\in\Gamma}\max_{M\in\Gamma}f(M,\lambda)=2+2^\lambda$$ for $\lambda\in[0;2]\cup[4;\infty)$.
\item $$\min_{A,B,C\in\Gamma}\max_{M\in\Gamma}f(M,\lambda)=2\sqrt3^\lambda$$ for $\lambda\in(2;4)$.
\item $$\max_{A,B,C\in\Gamma}\min_{M\in\Gamma}f(M,\lambda)=2+2^\lambda$$ for $\lambda<0$.
\end{enumerate}

Note that in the last case the order of the maximum and the minimum is reversed. We are not interested in the maximum of the function $f(M,\lambda)$ when $\lambda<0$ as it is infinity when $M\rightarrow A$,$B$ or $C$.\\

Prior to the present article the exact extremal values of $f(M,\lambda)$, established in \cite{Pacific}, were only known for $\lambda\in[0;2]$. We also give another (more elementary) proof of the results obtained in the previous article.

The question of the extremal behavior of the function $$f_n(M,-2)=\sum_{i=1}^nMM_i^{-2}$$ is considered in \cite{Chebyshev}, where it is proved that there always exist a point $M\in\Gamma$ such that $$f_n(M,-2)\leq\frac{n^2}{4}$$ (see also \cite{Tamas} for a short proof). This bound is sharp if and only if $\{M_i:i=1 \ldots n\}$ are the vertices of a regular $n$-gon. This agrees with our results when $n=3$.

We also consider the case when there are $n$ points $P_i, i=1\ldots, n$ on the unit circle, which are the vertices of a regular $n$-gon and a point $X$ on a circle $\Gamma$, concentric to the circumscribed circle of $P_1\ldots P_n$. We study the extremal values of$$R_n(X,\lambda)=\sum_{i=1}^nXP_i^{\lambda},$$
where $X\in\Gamma$.

This problem has been considered by Stolarsky \cite{Pacific}, who solves it for $0\leq\lambda< 2n$, when $\Gamma$ is the circumscribed circle of the polygon, and by Mushkarov \cite {Mushkarov}, who finds $\lambda$ for which the sum does not depend on the position of $X$ on $\Gamma$, again when $\Gamma$ is circumscribed around $P_1\ldots P_n$, and gives a trigonometric representation for higher powers.
In this paper we characterize the extremal behavior of the sum  $\sum_{i=1}^nXP_i^{\lambda}$ and prove the following theorem:

\begin{theorem}
Let $P_i, i=1,\ldots,n$ be the vertices of a regular $n$-gon inscribed in the unit circle. Now let $\Gamma$ be a circle concentric to the circumscribed circle. Put $B_i=OP_i\bigcap\Gamma$, where $O$ is the center of the $n$-gon.

Let $X\in\Gamma$ and $$R_n(X,\lambda)=\sum_{i=1}^nXP_i^{\lambda}.$$
\begin{enumerate}
\item $\lambda<0$. The minimum of $R_n(X,\lambda)$ is achieved when $X$ bisects the arc between consecutive vertices of $B_1\ldots B_n$ and the maximum when $X\equiv B_i$. In the case when $\Gamma$ is the circumscribed circle around $P_1\ldots P_n$ this function is not bounded($X\to B_i$ for some $i$).

\item $0\leq\lambda< 2n$. If  $\lambda$ is an even integer then $R_n(X,\lambda)$ is independent of the position of $X$ on $\Gamma$.\\ Otherwise let $m$ be such an integer, that $2m\leq\lambda\leq2m+2$.\\ If $m$ is even(odd) then $R_n(X,\lambda)$ is maximal(minimal) if and only if $X$ bisects the arc between consecutive vertices of $B_1\ldots B_n$. Moreover $R_n(X,\lambda)$ is minimal(maximal) if and only if $M\equiv B_i$.
\item $2n\leq\lambda$. If $n$ is even(odd) the maximum(minimum) of $R_n(X,\lambda)$ is obtained when $X$ coincides with one of the vertices of $B_1\ldots B_n$  and the minimum(maximum) is achieved when $X$ bisects the arc between consecutive vertices.
\end{enumerate}
\end{theorem}

\begin{remark}{\rm A case of part 2 of the theorem is proved in \cite{Pacific}, when $\Gamma$ is the circle circumscribed around $P_1\ldots P_n$. However is seems possible that the general result of part 2 can be proved in the same manner.

It is easy to see that part 3 of the theorem is actually true for $\lambda>2n-2$.}
\end{remark}

We first begin with the consideration of the regular $n$-gon as we will use this result later.
\medskip

\noindent{\bf Acknowledgement.} The authors would like to thank the referee for his comments and remarks, which helped improve the clarity of the paper.

\section{Regular $n$-gon}

We say that $y$ is a root of degree $k$ of an equation $f(x)=0$, where $f$ is $k-$times differentiable, if $f(y)=0$ and $f^t(y)=0$ for $t=1,\ldots,k-1$ and $f^k(y)\neq0$, where $f^t(x)$ denotes the $t$-th derivative of $f$.

We begin with the following lemma.

\begin{lemma} Let $a_1, a_2,\ldots, a_n$  and $b_1, b_2,\ldots, b_n$ be real numbers and $b_i, i=1,\ldots, n$ be nonnegative, then $$\sum_{i=1}^na_ib_i^{\lambda}=0$$ is either identically zero or has at most $n-1$ real solutions for $\lambda$ counted with their multiplicities.
\end{lemma}
\begin{proof}
We proceed by induction on the number of summands.
For $n=1$ we have that $ab^{\lambda}=0$, which does not have solutions if both of $a$ and $b$ are nonzero. If either of them is zero then $ab^{\lambda}$ is identically zero.
Now assume the statement to be true for all $k<n$.
For $k=n$ if either of $a_i$ or $b_i$ is zero then we use the induction hypothesis.

Now let $b_i, a_i$ be nonzero. As all of $b_i$ are nonzero then we can divide each term by $b_1^{\lambda}$ to get $$\sum_{i=1}^na_i\Big(\frac{b_i}{b_1}\Big)^{\lambda}=0.$$

Assume that this equation is not identically zero and its solutions are $y_1,\ldots,y_k$ with multiplicities $t_1,\ldots, t_k$ and $\sum_{i=1}^k t_i>n-1$.

Differentiating this with respect to $\lambda$ we get $$\sum_{i=2}^na_i\ln\Big(\frac{b_i}{b_1}\Big)\Big(\frac{b_i}{b_1}\Big)^{\lambda}=0=\sum_{i=2}^n a_i^{'} b_i^{'\lambda},$$
where $a_i'=a_i\ln\Big(\frac{b_i}{b_1}\Big)$ and $b_i'=\frac{b_i}{b_1}$. Assume that this expression is identically zero, then $\sum_{i=1}^na_ib_i^{\lambda}=0$ must be a constant, and the claim follows. Assume that the derivative does not vanish for all $\lambda$.
Now by the induction hypothesis the derivative has at most $n-2$ zeros.
But we have that $y_1,\ldots,y_k$ are solutions to the above equation with multiplicities $t_1-1,\ldots, t_k-1$, moreover by Rolle's theorem the derivative has at least one root in each interval $(y_i;y_{i+1})$, and thus we obtain $k-1+\sum_{i=1}^k t_i-1$ solutions (counted with their multiplicities), which is greater than $n-2$- a contradiction. It follows that $\sum_{i=1}^k t_i\leq n-1$. The lemma is proved.
\end{proof}

We continue with another problem, which is a part of Theorem 2.

\begin{theorem}
Let $P_1,\ldots, P_n$ be the vertices of a regular polygon, given a circle $\Gamma$, concentric to the circle circumscribed around $P_1P_2\ldots P_n$, then $\sum_{i=1}^n PP_i^{2k}$ is independent of the position of $P\in\Gamma$ for $k\in\{1,\ldots,n-1\}$.

\end{theorem}

\begin{proof}

We step to the use of complex numbers. We may assume that the circumscribed circle around $P_1\ldots P_n$ is the unit circle and that the radius of $\Gamma$ is $R$.
Let us assign to the vertices of the $n$-gon the complex numbers $\xi, \xi^2,\ldots, \xi^n$, where $\xi$ is a primitive $n-$th root of unity. We wish to prove that $$\sum_{i=1}^n|x-\xi^i|^{2k}=\mbox{const}$$ for all $x$ with a fixed norm $R$ and all $k\in\{1,\ldots,n-1\}$. We have
$$\sum_{i=1}^n|x-\xi^i|^{2k}=\sum_{i=1}^n(x-\xi^i)^k(\overline{x-\xi^i})^{k}=\sum_{i=1}^n(x-\xi^i)^k(\frac{R}{x}-\frac{1}{\xi^i})^k$$
After multiplying out we obtain $$(x-\xi^i)^k(\frac{R}{x}-\frac{1}{\xi^i})^k=\sum_{j=-k}^kc_j\xi^{-ij}x^j= P_i(x)$$ for all $x$ with $|x|=R.$
We now have $$\sum_{i=1}^n|x-\xi^i|^{2k}=\sum_{i=1}^nP_i(x)=\sum_{j=-k}^k \sum_{i=1}^nc_j\xi^{-ij}x^{j},$$
but $\sum_{i=1}^n\xi^{-ij}=0$ for all $j$ except $j=mn$, where $m$ is an integer, so  $$\sum_{i=1}^n|x-\xi^i|^{2k}=nc_0.$$
\end{proof}

\begin{remark}
{\rm One can prove that this is a characteristic property of the regular $n$-gon. That is: Given $n$ different points in the plane $A_1,\ldots, A_n$ and a circle $\Gamma$, such that $\sum_{i=1}^n PA_i^{2k}$ is independent of the position of $P$ on $\Gamma$ for every $k\in\{1,\ldots,n-1\}$, then these points are the vertices of a regular $n$-gon. It is conjectured that this remain true if the condition holds only for $k=2n-2$ and this has been verified for $n=3$ and $n=4$, but the authors have no proof for higher values of $n$}
\end{remark}

We are now ready to begin with the proof of Theorem 2.

\begin{proof}Due to symmetry we need only consider the case when $X\in  \widehat{B_1M}$, where $M$ is the midpoint of the arc $\widehat{B_1B_2}$.

After we position ourselves in a Cartesian coordinate system W.L.O.G we can assume that $P_1$ has coordinates $(1,0)$. Thus the coordinates of $P_i$ are $(\cos((i-1)2\pi/n), \sin((i-1)2\pi/n))$ and $X$ has coordinates $(a\cos x, a\sin x)$, where $x\in[0;2\pi/n]$.

We can now write the sum $$\sum_{i=1}^n|P_iX|^{\lambda}=\sum_{i=1}^n((a\cos x-\cos((i-1)2\pi/n))^2+(a\sin x-\sin((i-1)2\pi/n))^2)^{\frac{\lambda}{2}}=F(x,\lambda).$$

We differentiate this with respect to $x$ to obtain
$$\frac{\partial F(x,\lambda)}{\partial x}=\sum_{i=1}^n\lambda |P_iX|^{\lambda-1}\frac{d |P_iX|}{d x}.$$
The partial derivative exists for $x\in(0;2\pi/n)$.
Now fix $x$ and consider this as a function of $\lambda$. As we are interested only in the sign of the derivative we can consider only $\sum_{i=1}^n|P_iX|^{\lambda-1}\frac{d |P_iX|}{d x}$ for $\lambda\neq0$.

As we have proved earlier $F(x,\lambda)$ is constant for $\lambda=2,4,\ldots, 2n-2$ and so $\frac{\partial F(x,\lambda)}{\partial x}$ vanishes for these values of $\lambda$. But from Lemma 1 this expression is either identically zero or has at most $n-1$ solutions for $\lambda$, counted with their multiplicities.

We shall prove that this expression as a function of $\lambda$ is not identically zero for fixed $x\in(0;2\pi/n)$ . For sake of contradiction assume otherwise.
Let $x\in(0; 2\pi/n)$. It is now easy to see that for the point $X$ corresponding to this $x$ the distances $|P_iX|$ are all different. Now take $|P_iX|=\max\{|P_iX|;\frac{d |P_iX|}{d x}\neq0\}$, if $|P_iX|>1$ then $\lim_{\lambda\to\infty}|\frac{\partial |F(x,\lambda)}{\partial x}|=\infty$. If $|P_iX|\leq1$, then let $|P_iX|=\min\{|P_iX|;\frac{d |P_iX|}{d x}\neq0\}$ and it follows that $|P_iX|<1$ as otherwise we have that there must be two distances $P_iX$ that are equal or that $n-1>0$ of $\frac{d |P_iX|}{d x}=0$, which is not possible. This is to the fact that $|P_iX|$ is increasing when $X$ travels one of the arcs $\widehat{P_i^{'}P_i^{''}}$ or $\widehat{P_i^{''}P_i^{'}}$ and decreasing on the other one, where $P_i^{'}, P_i^{''}$ are the two intersections of the line trough $O$ and $P_i$ with $\Gamma$. It is obvious that $P_i^{'}$ and $P_i^{''}$ either coincide with some of $B_i$ or are midpoints of some arc between consecutive vertices of $B_1\ldots B_n$.

Now again considering $\lim_{\lambda\to-\infty}|\frac{\partial |F(x,\lambda)}{\partial x}|=\infty$ we obtain the desired result.

As we have mentioned for a fixed $x$, $\frac{\partial F(x,\lambda)}{\partial x}=0$ for every $\lambda=0,2,4,\ldots, 2n-2$, from where it follows that these are all the solutions for $\lambda$ and each of them (except possibly $\lambda=0$) must have multiplicity on. Moreover the derivative changes sign at $\lambda=0$. This means that for a fixed $x$, $\frac{\partial F(x,\lambda)}{\partial x}$ changes sign at $\lambda=0,2,4,\ldots, 2n-2$.
Now assume that for some $\lambda_0\neq 0,2,4,\ldots, 2n-2$ there exist $y$ and $z$ in the interval $(0;2\pi/n)$, such that
$$\frac{\partial F(y,\lambda_0)}{\partial x}\frac{\partial F(z,\lambda_0)}{\partial x}<0,$$

then as $\frac{\partial F(x,\lambda_0)}{\partial x}$ is a continuous function of $x$ then there is $t\in(0;2\pi/n)$ such that $\frac{\partial F(t,\lambda_0)}{\partial x}=0$, and then it follows that $\frac{\partial F(t,\lambda)}{\partial x}=0$ for all $\lambda$, which is a contradiction.
Hence we have that the derivative $$\frac{\partial F(x,\lambda)}{\partial x}$$ does not change when $x\in(0;2\pi/n)$ for fixed $\lambda$, also for every fixed $x\in(0,2\pi/n)$ it changes sign at $\lambda=0,2,\ldots 2n-2$. Thus as $F(x,\lambda)$ is a continuous function of $x$ we have obtained that the minimum and maximum of that function for $x\in[0;2\pi/n]$ are obtained  when $x=0$, or $x=2\pi/n$.

Now consider $$\lim_{\lambda\to\infty}\frac{F(0,\lambda)}{F(2\pi/n,\lambda)}.$$

Assume that $n$ is even, then $|B_1P_{n/2+1}|>|MP_i|$ for every $i$ and then the above limit is $\infty$.
Assume that $n$ is odd, then $|MP_{\left\lceil n/2\right\rceil }|>|B_1P_i$ for every $i$ and the above limit becomes $0$.
This proves part 3 of Theorem 2.
Now taking into account parity and the above observations for the intervals in which $\frac{\partial F(x,\lambda)}{\partial x}$ changes sign the conclusion of the theorem easily follows.
\end{proof}

\begin{remark}{\rm When $\Gamma$ is the circumcircle of the regular polygon part 1 of Theorem 2 is easily proved by the observation that each of the functions $MP_i^{\lambda}+MP_{n+1-i}^{\lambda}$ is concave.}
\end{remark}

\section{Consideration of the case for three base points}

\subsection{Proof of the case $\lambda<0$}
We now consider the case $\lambda<0$.

Let $\angle{C}=\max\{\angle{A},\angle{B}, \angle{C}\}$ and $M$ be the midpoint of the smaller arc $\widehat{AB}$. We shall prove that $f(M,\lambda)\leq 2+2^\lambda$. We consider two cases:
\\
1.The angle $\angle{C}\geq\pi/2$. Then the maximum of the function $MA^\lambda+MB^\lambda+MC^\lambda$ when $C$ travels along the smaller arc $\widehat{AB}$ is achieved when $C\equiv{A}$ or when $C\equiv{B}$, as $MD>MA=MB$ for any point $D$ on the smaller arc $\widehat{AB}$. Now we have $f(M,\lambda)=3MA^{\lambda}\leq3\sqrt2^\lambda<2^\lambda+2$

2.The angle $\angle{C}=x<\pi/2$. Then $\angle{C}\in[\pi/3; \pi/2)$. Now let $C'$ and $C''$ be the points for which $\angle{ABC'}$ and $\angle{BAC''}$ respectively equal $x$. It is easy to see that $C$ belongs to the smaller arc $\widehat{C'C''}$ as $\angle{C}$ is the largest angle of the triangle. It is also easy to see that the maximum of $f(M,\lambda)$ when $C$ belongs to the arc $\widehat{C'C''}$ is obtained exactly when $C\equiv{C'}$ or $C\equiv{C''}$ as $MC''=MC'\leq MC$ for every $C$ on $\widehat{C'C''}$. Without loss of generality we can assume that $C\equiv C'$. Then we can express the function $f(M,\lambda)=2(2\sin(x/2))^\lambda+(2\sin{\frac{3}{2}x})^\lambda=2^\lambda(2\sin^\lambda\frac{x}{2}+\sin^\lambda\frac{3}{2}x)=F(x,\lambda)$. We differentiate with respect to $x$ to get $$\frac{\partial{F(x,\lambda)}}{\partial{x}}= \lambda2^\lambda(\sin^{\lambda-1}\frac{x}{2}\cos\frac{x}{2}+\frac{3}{2}\sin^{\lambda-1}\frac{3}{2}x\cos\frac{3}{2}x)$$
It is now easy to see that both $\sin^{\lambda-1}\frac{x}{2}\cos\frac{x}{2}$ and $\frac{3}{2}\sin^{\lambda-1}\frac{3}{2}x\cos\frac{3}{2}x$ are decreasing functions in the interval $[\pi/3; \pi/2)$  as $\lambda<0$. \\
Then $\frac{\partial F(x,\lambda)}{\partial x}$ is an increasing function of $x$ in this interval as $\lambda<0$, hence $F(x,\lambda)$ is a convex function of $x\in[\pi/3;\pi/2)$ if $\lambda<0$.

From here it follows that $\sup_{\substack{x\in[\pi/3;\pi/2)}}F(x,\lambda)=\max\{F(\pi/3,\lambda),\lim_{\substack{x\to\pi/2}}F(x,\lambda)\}=\max\{2+2^\lambda,3\sqrt2^\lambda\}$ as $F(x,\lambda)$ is continuous in the interval $[\pi/3;\pi/2]$. Now we have that $\max\{2+2^\lambda,3\sqrt2^\lambda\}=2+2^\lambda$ and this bound is achievable only for $ABC$ an equilateral triangle, for all other configurations of the points $ABC$ the function $MA^\lambda+MB^\lambda+MC^\lambda<2+2^\lambda$ for the specified point $M$.

Now using Theorem 2 we get that the minimum of $MA^{\lambda}+MB^{\lambda}+MC^{\lambda}$ is obtained when $M$ bisects the arc between consecutive vertices of the triangle, and in this case we have $MA^{\lambda}+MB^{\lambda}+MC^{\lambda}=2+2^{\lambda}$. This concludes the proof.

\begin{remark}{\rm This case can also be proved using the main approach of \cite{Chebyshev}. It is based on the fact that the local minima on each of the arcs between consecutive
base points must be equal for all $\lambda<0$ (Lemma 1 therein). In the case of only three points one can obtain that the equilateral triangle is indeed the extremal case. Assume otherwise.
It is not difficult to see that it is not possible all of the local minimums to be equal when two of the points are closer than $\sqrt{2}$. Assume now that $C$ is not the midpoint of the arc $\widehat{AB}$. If we consider the function $f_1=|MA|^{\lambda}+|MB|^{\lambda}+|MC_1|^{\lambda}$, where $C_1$ is the midpoint of $\widehat{AB}$. We may assume that $C$ belongs to the shorter arc $C_1B$. Due to symmetry the local minima of $f_1$ are equal on the short arcs $\widehat{AC_1}$ and $\widehat{BC_1}$.
Now we have that $\angle{C_1OC}\leq\pi/4$. Thus it follows that $f>f_1$ on $\widehat{AC_2}$ and $f<f_1$ on $\widehat{BC_2}$, where $C_2$ is the midpoint of the arc $\widehat{CC_1}$. But from here we obtain that the local minima of $f$ cannot be equal on the shorter arcs $\widehat{AC}$ and $\widehat{BC}$.
\footnote{The authors would like to thank the referee for suggesting this approach.}}
\end{remark}

\subsection{Proof of the case $\lambda>2$}

We first prove that for every three points $A,B$ and $C$ on the unit circle there exists a point $M$ also on the unit circle,  such that $f(M,\lambda)\geq\max\{2+2^\lambda, 2(\sqrt3)^\lambda\}$

Let $AB= \min\{AB,BC,CA\}$ now let the bisector of $AB$ intersect the unit circle $\Gamma$ at $M'$. Then $\angle{BAM'}=\angle{ABM'}=x$ and $\pi/3\leq x<\pi/2$. Now by the sine rule $BM'=AM'=2\sin x\geq\sqrt3$ and we have $f(M',\lambda)\geq2(\sqrt3)^\lambda$ with equality only if $x=\pi/3$ or equivalently if $ABC$ is an equilateral triangle.\\

It remains to prove that for every triangle there is a point $M'$ such that $f(M',\lambda)\geq2+2^\lambda$. We consider two cases-when $ABC$ is obtuse-angled and when it is acute-angled.\\

1. Let $\angle{C}=\max\{\angle{A},\angle{B},\angle{C}\}\geq\pi/2$, also let $O$ be the center of $\Gamma$. Let $M'=CO\cap\Gamma$. We have $CM'=2$ and $f(M',\lambda)=2^\lambda+BM'^\lambda+AM'^\lambda$. Now $\angle{BAM'},\angle{ABM'}
\leq\pi/2$ and $\angle{BAM'}+\angle{ABM'}\geq\pi/2$. \\
We have $\pi/4\leq\max\{\angle{BAM'}\angle{ABM'}\}\leq\pi/2$ and so $BM'^\lambda+AM'^\lambda>(2.1/\sqrt{2})^\lambda\geq 2$ as $\lambda\geq 2$, so $f(M',\lambda)>2+2^\lambda.$ And this bound cannot be achieved for an obtuse-angled triangle.

2.Let $c=\angle{C}=\max\{\angle{A},\angle{B},\angle{C}\}<\pi/2$. Now $M'=CO\cap\Gamma$. We have $CM'=2$ and $f(M',\lambda)=2^\lambda+BM'^\lambda+AM'^\lambda$. We shall prove that $BM'^\lambda+AM'^\lambda\geq2$. Let $\angle{ACM'}=x$. By the sine rule $
BM'^\lambda+AM'^\lambda$=$(2\sin{x})^\lambda+(2\sin{c-x})^\lambda=f_1(x,\lambda)$.
We shall prove that $\sin^\lambda{x}+\sin^\lambda{(c-x)}>2\sin^\lambda{(c/2)}$. We have that $$\frac{\partial{f_1(x,\lambda)}}{\partial{x}}=\lambda2^\lambda(\sin^{\lambda-1}x\cos{x}-\sin^{\lambda-1}(c-x)\cos(c-x))=$$$$=\lambda2^{\lambda-1}(\sin^{\lambda-2}x\sin{2x}-\sin^{\lambda-2}(c-x)\sin(2c-2x)).$$
It is now easy to see that for $x\in[0;c/2)$    $\sin^{\lambda-2}x<\sin^{\lambda-2}(c-x)$ $(\lambda\geq2)$ and $\sin2x<\sin(2c-2x)$ and so  $$\frac{\partial{f_1(x,\lambda)}}{\partial{x}}<0$$ for $x\in[0;c/2)$. With analogous arguments it follows that

$$\frac{\partial{f_1(x,\lambda)}}{\partial{x}}=0$$ for $x=c/2$.
$$\frac{\partial{f_1(x,\lambda)}}{\partial{x}}>0$$ for $x\in(c/2;c]$.\\

Then $\min_{\substack{x\in[0;c]}}(f_1(x,\lambda)+f_1(c-x,\lambda))=2f_1(c/2,\lambda)\geq2f_1(\pi/6)=2$. Equality holds iff $ABC$ is equilateral.\\

We have that $2+2^\lambda=2\sqrt3^\lambda$  for $\lambda\in\{2,4\}$, $2+2^\lambda<2\sqrt3^\lambda$ for $\lambda\in(2;4)$ and $2+2^\lambda>2\sqrt3^\lambda$
for $\lambda>4$.
We shall prove that for $ABC$ an equilateral triangle those bounds are sharp.\\

Again using Theorem 2 we get:
\begin {enumerate}
\item When $\lambda\in[2;4]$ we have that the maximum of $MA^{\lambda}+MB^{\lambda}+MC^{\lambda}$ is achieved when $M$ coincides with one of $A, B, C$ and is equal to $2\sqrt3^{\lambda}$.
\item When $\lambda>4$ we have that the maximum of $MA^{\lambda}+MB^{\lambda}+MC^{\lambda}$ is achieved when $M$ bisects the arc between consecutive vertices of the triangle $ABC$ and is equal to $2+2^{\lambda}$.
\end{enumerate}
These bounds are sharp.

We also have the minimum of $f(M,\lambda)$ when $ABC$ is an equilateral triangle. Namely when $\lambda\in[2,4]$ $\min{f(M,\lambda)}=2+2^\lambda$ and $\min{f(M,\lambda)}=2\sqrt3^\lambda$ when $\lambda>4$.
This concludes the proof.

\subsection{Proof of the case $\lambda\in[0;2]$}

This case is proved in \cite{Pacific}, but nevertheless we give a new proof, independent of the mentioned article.

We shall now prove that for every three points $A$, $B$, $C$ on the unit circle and a real number $\lambda\in(0;2)$ there exists a point $M$ again on the unit circle, such that $f(M,\lambda)=MA^\lambda+MB^\lambda+MC^\lambda\geq 2+2^\lambda$ and this bound is sharp. It is only achievable when $A$,$B$ and $C$ are the vertices of a equilateral triangle.

Let again $\angle{C}=\max\{\angle{A},\angle{B},\angle{C}\}=x$. As before it is easy to see that when $\angle{C}\geq\pi/2$ when we choose $M$ to be the midpoint of the arc $\widehat{AB}$, $f(M,\lambda)$ is greater than or equal to $3\sqrt{2}^\lambda\geq2+2^\lambda$ for $\lambda\in[0;2]$. When $\lambda\in\{0,2\}$ we have that $MA^\lambda+MB^\lambda+MC^\lambda$ is constant.
We can assume that the triangle $ABC$ is acute-angled. Then $\angle{C}\in[\pi/3; \pi/2]$ and again let $M$ be the midpoint of the arc $\widehat{AB}$.
Now let $C'$ and $C''$ be the points for which $\angle{ABC'}$ and $\angle{BAC''}$ respectively equal $x$. It is easy to see that $C$ belongs to the smaller arc $\widehat{C'C''}$ as $C$ is the largest angle of the triangle.  It is also easy to see that the minimum of $f(M,\lambda)$ when $C$ belongs to the arc $\widehat{C'C''}$ is obtained exactly when $C\equiv{C'}$ or $C\equiv{C''}$ as $MC''=MC'\leq MC$ for every $C$ on $\widehat{C'C''}$.

Let $\angle{A}=x$. Now using the sine rule we get $$f(M,\lambda)=MA^\lambda+MB^\lambda+MC^\lambda=2\Big(2\sin\frac{x}{2}\Big)^\lambda+\Big(\sin\frac{3}{2}x\Big)^\lambda=F(x,\lambda).$$
We shall prove that if $x\in[\pi/3;\pi/2]$  then $F(x,\lambda)\geq2+2^\lambda$.

We consider two cases-$\lambda\in(0;1)$ and $\lambda\in[1;2)$.

Let $\lambda\in(0;1)$.

After differentiating with respect to $x$ we get:
$$\frac{\partial{F(x,\lambda)}}{\partial{x}}=\lambda2^\lambda(\sin^{\lambda-1}\frac{x}{2}\cos\frac{x}{2}+ \frac{3}{2}\sin^{\lambda-1}\frac{3}{2}x\cos\frac{3}{2}x).$$\label{equation}\\

It is now easy to see that both $\sin^{\lambda-1}\frac{x}{2}\cos\frac{x}{2}$ and $\sin^{\lambda-1}\frac{3}{2}x\cos\frac{3}{2}x$ are decreasing functions as $\lambda-1<0$. Then we have that $F(x,\lambda)$ is a concave function of $x$ when $\lambda\in(0;1)$. It follows that $$\min_{x\in[\pi/3;\pi/2)}{F(x,\lambda)}=\min\{F(\pi/3,\lambda),\lim_{x\to\pi/2}F(x,\lambda)\}=F(\pi/3;\lambda)=2+2^\lambda$$
and for every $x\neq\pi/3$ we have $F(x,\lambda)>2+2^\lambda$. We shall later prove that for $A$,$B$,$C$ the vertices of an equilateral triangle this bound is sharp.\\
Now let $\lambda\in[1;2)$. Let $CO\cap\Gamma=M$. We shall prove that for the point $M$ we have $AM^\lambda+BM^\lambda+CM^\lambda\geq 2+2^\lambda$. We have that $CM=2^\lambda$. We only need to prove that $BM+CM\geq2$ as we have that $$BM^\lambda+AM^\lambda\geq2\Big(\frac{AM+BM}{2}\Big)^\lambda.$$

\begin{lemma}Let $A$,$B$ and $C$ be points on the unit circle $\Gamma$ with center $O$. Assume $\angle{C}=\max\{\angle{A},\angle{B},\angle{C}\}$ and $M=CO\cap\Gamma$ then $MA+MB\geq2$.
\end{lemma}
\begin{proof}
We have that $MA+MB=2(\sin{x}+\sin(c-x))=f(x)$, where $x=\angle{MAC}$ and $c=\pi-\angle{ACB}<\pi/2$. We have that $f'(x)=2(\cos{x}-\cos(c-x))$ and $f'(x)>0$ for $x\in[0;c/2)$, $f'(x)=0$ for $x=c/2$ and $f'(x)<0$ for $x\in(c/2;c]$.\\

Now let $C'$ and $C''$ be the points for which $\angle{ABC'}$ and $\angle{BAC''}$ respectively equal $x$. It is easy to see that $C$ belongs to the smaller arc $\widehat{C'C''}$ as $C$ is the largest angle of the triangle.
We have that $\min MA^{\lambda}+MB^{\lambda}$ is obtained when $MO\cap\Gamma=C'$ or $MO\cap\Gamma=C''$ as $f(x)$ is concave. Now $\angle{ABC}=\angle{BCA}=\gamma$. Then $MA+MB=2(\sin(\pi/2-\gamma)+\sin(2\gamma-\pi/2))=\cos{\gamma}-\cos{2\gamma}=f_1(\gamma)$. Differentiating $f_1(\gamma)$ we get $f_1'(\gamma)=2\sin{2\gamma}-\sin{\gamma}$ which is a decreasing function of $\gamma\in[\pi/3;\pi/2)$. This gives us that $f_1(\gamma)$ is a concave function when $\gamma\in [\pi/3;\pi/2)$ and it follows that $$\min f_1(\gamma)=\min_{\gamma\in[\pi/3;\pi/2)}\{f_1(\pi/3),\lim_{x\to\pi/2}f_1(x)\}=2.$$
\end{proof}

Now $BM^\lambda+AM^\lambda\geq2\Big(\frac{AM+BM}{2}\Big)^\lambda\geq2$ with equality only when $AM=BM=1$ which is possible only when $A$,$B$ and $C$ are the vertices of an equilateral triangle.

In such a way we obtain that when $\lambda\in(0;2)$ there exists a point $M$ on the unit circle, such that $MA^\lambda+MB^\lambda+MC^\lambda\geq2+2^\lambda$ and this bound is achievable only if $A$, $B$ and $C$ are the vertices of an equilateral triangle.\\

Now using again the result of Theorem 2 one easily obtains that the maximum of $MA^{\lambda}+MB^{\lambda}+MC^{\lambda}$ is obtained when $M$ is the midpoint of one of the arcs between consecutive vertices and it indeed equals $2+2^{\lambda}$.

\end{document}